\documentclass{amsart}
\usepackage{amsmath}
\usepackage{amsfonts}
\usepackage{amssymb}
\usepackage{graphicx}%

\newtheorem{myrem}{Remark}[section]
\newtheorem{mytheo}{Theorem}[section]
\newtheorem{mylem}{Lemma}[section]

\newtheorem{mycor}{Corollary}[section]

\usepackage{color}

\theoremstyle{remark}

\numberwithin{equation}{section}

 \allowdisplaybreaks[4]
\begin{document}

\title{Deviation Inequalities for the Spectral Norm of Structured Random Matrices}


\author{Guozheng Dai}
\address{Zhejiang University, Hangzhou, 310027,  China.}
\email{11935022@zju.edu.cn}

\author{Zhonggen Su}
\address{Zhejiang University, Hangzhou, 310027,  China.}
\email{suzhonggen@zju.edu.cn}

\date{}

\keywords{comparison of weak and strong moments, contraction principle, deviation inequality, spectral norm
}

\begin{abstract}
We study the deviation inequality for the spectral norm of structured random matrices with non-gaussian entries. In particular, we establish an optimal bound for the $p$-th moment of the spectral norm by transfering the spectral norm into the suprema of canonical processes. A crucial ingredient of our proof is a comparison of weak and strong moments. As an application, we show a deviation inequality for the smallest singular value of a rectangular random matrix.
\end{abstract}

\maketitle

\section{Introduction and Main Result}
Random matrix theory has been an important area of probability theory. There has been much work to investigate precise analytic results and limit theorems of random matrices with exact or approximate symmetrices \cite{Anderson_book_matrices, Tao_book_matrices}. More recently, there has also been considerable interest in structured random matrices whose entries are no longer identically distributed. In particular, the structured random matrix $X$ is an $n\times n$ symmetric random matrix with entries $X_{ij}=b_{ij}\xi_{ij}$, where $\{\xi_{ij}, 1\le j\le i\le n \}$ are independent standard random variables and $\{b_{ij}, 1\le j\le i\le n  \}$ are given nonnegative scalars.
This model includes many interesting cases, such as the Wigner matrix, which is the case $b_{ij}=1$ for all $1\le i, j\le n$, random band matrices where $b_{ij}=\mathbb{I}_{\vert i-j\vert\le k}$ with a constant $k\ge 0$, and so on.

Understanding the behavior of the spectral norm of random matrices is a fundamental question in random matrix theory and one of considerable importance in many modern applications \cite{Tropp_book}.
In this paper, we pay attention to the structured random matrix's spectral norm $\Vert X\Vert$, and particularly focus on deviation inequalities of $\Vert X\Vert$. For the sake of providing some background, we shall give a very simple review of the relevant results.

Let $X=(b_{ij}\xi_{ij})$ as above and assume further $\xi_{ij}\sim N(0, 1)$. As a consequence of the noncommutative Khintchine inequality \cite{Pisier_book}, we have the following useful nonasymptotic bound on $\Vert X\Vert$ in our setting
\begin{align}
	\textsf{E}\Vert X\Vert \lesssim \sigma \sqrt{\log n}, \quad \text{with}\,\, \sigma:=\max_{i}\sqrt{\sum_{j}b_{ij}^{2}},\nonumber
\end{align}
where we write $a\lesssim b$ if $a\le Cb$ for a universal constant $C$.
Unfortunately, this bound is not optimal even in the most common case of Wigner matrices ($b_{ij}=1$ for all $i, j$). Indeed, $\sigma=\sqrt{n}$ in the Wigner case. Hence, the resulting bound $\textsf{E}\Vert X\Vert\lesssim \sqrt{n\log n}$ falls short of the correct scaling $\textsf{E}\Vert X\Vert\sim \sqrt{n}$.

Another bound on $\Vert X\Vert$ can be deduced from a simple $\varepsilon$-net argument; see Theorem 4.4.5 in \cite{Vershynin_book_highprobability}. This yields the following inequality:
\begin{align}
	\textsf{E}\Vert X\Vert \lesssim \sigma_{*} \sqrt{ n}, \quad \text{with}\,\, \sigma_{*}:=\max_{i}\vert b_{ij}\vert.\nonumber
\end{align}
Although this bound is sharp when $X$ is a Wigner matrix, it becomes terrible in the diagonal case where $b_{ij}=\mathbb{I}_{\vert i-j\vert= 0}$, a special case of random band matrices. In particular, $\sigma_{*}=1$ and $\textsf{E}\Vert X\Vert=\textsf{E}\max_{i}\vert \xi_{ii}\vert\sim\sqrt{\log n}$ in this case.

By comparing the spectral norm $\Vert X\Vert$ with the quantity $(\text{Tr}(X^{p}))^{1/p}$ for $p\sim\log n$, where $\text{Tr}(A)$ is the trace of matrix $A$, Bandeira and van Handel \cite{Bandeira_aop} showed the following better bound: 
\begin{align}\label{Eq_bound_Banderia}
	\textsf{E}\Vert X\Vert\lesssim \sigma+\sigma_{*}\sqrt{\log n},
\end{align}
where $\sigma$ and $\sigma_{*}$ are defined as above. One can easily verify that this result is optimal in the Wigner and diagonal cases mentioned above. Here, we say a bound for $\textsf{E}\Vert X\Vert$ is optimal if it captures $\textsf{E}\Vert X\Vert$ correctly up to universal constants.
The bound in \eqref{Eq_bound_Banderia} has been proven to be optimal in many settings (see Corollary 3.15 in \cite{Bandeira_aop}), except for some sparse cases,  such as the case $b_{ij}=0$. Note that in this case,
\begin{align}
	\textsf{E}\Vert X\Vert=\textsf{E} \max_{i}b_{ii}\vert\xi_{ii}\vert\asymp \max_{i}b_{ii}^{*}\sqrt{\log (i+1)}\le \max_{i}b_{ii}\sqrt{\log n},\nonumber
\end{align}
where $\{b_{ii}^{*} \}$ is a descending arrangement of $\{b_{ii} \}$. One can refer to Remark 4.6 in \cite{vanhandel_structure} for a detailed proof.

Only recently, did Latała, van Handel and Youssef \cite{Latala_Inventions} prove an optimal bound on $\Vert X\Vert$ for all cases,
\begin{align}
	\textsf{E}\Vert X\Vert\asymp \max_{i}\sqrt{\sum_{j}b_{ij}^{2}}+\max_{ij}b_{ij}^{*}\sqrt{\log i},\nonumber
\end{align}
where we write $a\asymp b$ if $a\lesssim b$ and $b\lesssim a$. Here the matrix $\{b_{ij}^{*}\}$ is obtained by permuting the rows and columns of the matrix $\{b_{ij}\}$ such that
\begin{align}
	\max_{j}b_{1j}^{*}\ge \max_{j}b_{2j}^{*}\ge \cdots\ge \max_{j}b_{nj}^{*}.\nonumber
\end{align}
They also extended their optimal result to the cases where $X_{ij}$ are non-gaussian. In particular, let $X=(b_{ij}\xi_{ij})$ be an $n\times n$ symmetric random matrix, where $\{b_{ij}, i\ge j \}$ are positive constants and  $\{\xi_{ij}, i\ge j \}$ are independent centered random variables such that
\begin{align}
	\theta_{1}p^{1/\alpha}\le (\textsf{E} \vert\xi_{ij}\vert^{p})^{1/p}\le \theta_{2}p^{1/\alpha} \quad \text{for all}\,\, p\ge 1,\nonumber
\end{align}
where $\alpha, \theta_{1}, \theta_{2}>0$. For convenience, we shall say such variables are $\alpha$-exponential with parameters $\theta_{1}, \theta_{2}$ in the sequel. Latała et al proved in the case $0<\alpha\le 2$
\begin{align}\label{Eq_optimal_bound}
	\textsf{E}\Vert X\Vert\asymp_{\theta_{1}, \theta_{2}, \alpha} \max_{i}\sqrt{\sum_{j}b_{ij}^{2}}+\max_{ij}b_{ij}^{*}(\log i)^{1/\alpha},
\end{align}
where $a\asymp_{\theta} b$ means $b/C(\theta)\le a\le C(\theta)b$ for a constant $C(\theta)$ only depending on $\theta$.

In this paper, we study the nonasymptotic properties of $\Vert X\Vert$ in the $\alpha$-exponential case, $\alpha>0$. Our main result shows an deviation inequality for $\Vert X\Vert$, reading as follows:
\begin{mytheo}\label{Theo_deviation}
	Assume $X=(b_{ij}\xi_{ij})$ is an $n\times n$ symmetric random matrix, where $\{b_{ij}, i\ge j \}$ are positive constants and  $\{\xi_{ij}, i\ge j \}$ are independent $\alpha$-exponential  random  variables with parameters $\theta_{1}, \theta_{2}$. Then, in the case $\alpha>0$, we have for $t\ge 0$
	\begin{align}
		\textsf{P}\big\{\Vert X\Vert\ge C_{1}\textsf{E}\Vert X\Vert +t\max_{ij}b_{ij} \big\}\le C_{2}\exp(-c_{1}t^{\alpha})\nonumber
	\end{align}
	and 
	\begin{align}
		\textsf{P}\big\{\Vert X\Vert\ge C_{3}\textsf{E}\Vert X\Vert +t\max_{ij}b_{ij} \big\}\ge C_{4}\exp(-c_{2}t^{\alpha}),\nonumber
	\end{align}
	where $c_{1}, c_{2}, C_{1}, \cdots, C_{4}$ are constants depending only on $\theta_{1}, \theta_{2}, \alpha$ and satisfying $C_{3}<C_{1}, C_{4}<C_{2}, c_{1}<c_{2}$.
\end{mytheo}
\begin{myrem}
	Let $A=(b_{ij}\xi_{ij})$ be an $N\times n$ asymmetric random matrix. Here $\{b_{ij}: 1\le i\le N, 1\le j\le n\}$ is a sequence of nonnegative constants and $\{ \xi_{ij}: 1\le i\le N, 1\le j\le n \}$ is a sequence of independent $\alpha$-exponential variables, $0<\alpha\le 2$. Recently, Adamczak et al  obtained the following deviation inequality (Proposition 1.16 in \cite{Adamczak}):
	\begin{align}\label{1}
		\textsf{P}\Big\{\Vert A\Vert\ge \max_{j\le n}\sqrt{\sum_{i\le N}b_{ij}^{2}}t  \Big\}\le Ce^{-ct^{\alpha}},
	\end{align}
	where $C, c$ are constants depending on $\theta_{1}, \theta_{2}, \alpha$. Set
	$$
	\hat{A}=\begin{bmatrix}
		0  &  A  \\
		A^{*}  &  0
	\end{bmatrix},$$
	where $A^{*}$ is the transpose of $A$.
	Then $\hat{A}$ is   symmetric  and  $\Vert \hat{A}\Vert=\Vert A\Vert$. Hence, Theorem \ref{Theo_deviation} yields that
	\begin{align}\label{2}
		\textsf{P}\{\Vert A\Vert\ge C_{1}\textsf{E}\Vert A\Vert+ t\max_{ij}b_{ij}  \}\le C_{2}e^{-c_{1}t^{\alpha}}.
	\end{align}
	With the result \eqref{Eq_optimal_bound}, our result \eqref{2} behaves better than \eqref{1}.
	To see this more clearly, consider the Wigner case (i.e. $b_{ij}=1$ for all $i, j$). Then \eqref{1} yields 
	\begin{align}
		\textsf{P}\{\Vert A\Vert\ge \sqrt{N}t  \}\le Ce^{-ct^{\alpha}}\nonumber
	\end{align}
	and \eqref{2} yields
	\begin{align}
		\textsf{P}\{\Vert A\Vert \ge \sqrt{N}t  \}\le C_{2}e^{-c_{3}(\sqrt{N}t)^{\alpha}}.\nonumber
	\end{align}
\end{myrem}

\section{Preliminaries}
\subsection{Notations}
For a fixed vector $x=(x_{1},\cdots, x_{n})^\top\in \mathbb{R}^{n}$, we denote by $\Vert x\Vert_{p}=(\sum\vert x_{i}\vert^{p})^{1/p}$ the $l_{p}$ norm. Set $S^{n}_{p}:=\{x\in\mathbb{R}^{n}: \Vert x\Vert_{p}=1 \}$ and $B^{n}_{p}:=\{x\in\mathbb{R}^{n}: \Vert x\Vert_{p}\le 1 \}$.
We use $\Vert \xi\Vert_{p}=(\textsf{E}\vert \xi\vert^{p})^{1/p}$ for the $L_{p}$ norm of a random variable $\xi$.
As for an $m\times n$ matrix $A=(a_{ij})$, we usually write $\Vert A\Vert$ as its spectral norm, i.e., $\Vert A\Vert=\sup_{x\in S^{n}_{2}}\Vert Ax\Vert_{2}$.

Unless otherwise stated, we denote by $C, C_{1}, c, c_{1},\cdots$ universal constants, and by $C(\delta), c(\delta),\cdots $ constants that depend only on the parameter $\delta$. 
For convenience, we write $f\lesssim g$ if $f\le Cg$ for some universal constant $C$ and write $f\lesssim_{\delta} g$ if $f\le C(\delta)g$ for some constant $C(\delta)$. We also say $f\asymp$ if $f\lesssim g$ and $g\lesssim f$, so does $f\asymp_{\delta} g$. 

We say $\xi$ is  a symmetric Weibull variable with the scale parameter $1$ and the shape parameter $\alpha$ if $-\log \textsf{P}\{\vert\xi\vert>x  \}=x^{\alpha}, x\ge 0$. For convenience, we shall write $\xi\sim \mathcal{W}_{s}(\alpha)$. Given a nonempty set $T\subset l_{2}:=\{t: \sum_{i}t_{i}^{2}<\infty \}$,  we call $\{S_{t}: t\in T \}$ a canonical process, where $S_{t}=\sum_{i=1}^{\infty}t_{i}\xi_{i}$ and $\{\xi_{i}\}$ are independent random variables.

\subsection{Tails and Moments}
In this subsection, we shall collect some properties about the tails  and the  moments of random variables. The following first lemma shows us how to deduce a tail bound from moment conditions. One can obtain this result by Markov's inequality and the Paley-Zygmund inequality.
\begin{mylem}[Lemma 4.6 in \cite{Latala_Inventions}]\label{Lem_moments1}
	Let $\xi$ be an $\alpha$-exponential random variable with parameters $\theta_{1}$ and $\theta_{2}$, i.e.,
	\begin{align}
		\theta_{1}p^{1/\alpha}\le (\textsf{E} \vert\xi\vert^{p})^{1/p}\le \theta_{2}p^{1/\alpha} \quad \text{for all}\,\, p\ge 2.\nonumber
	\end{align}
	Then there exist constants $c_{1}, c_{2}$ depending only on $\theta_{1}, \theta_{2}$ and $\alpha$ such that
	\begin{align}
		c_{1}e^{-t^{\alpha}/c_{1}}\le \textsf{P}\{\vert\xi\vert\ge t  \}\le c_{2}e^{-t^{\alpha}/c_{2}}\quad \text{for all}\,\, t\ge 0.\nonumber
	\end{align}
\end{mylem}
\begin{mylem}\label{Lem_Moments_2}
	Assume that a random variable $\xi$ satisfies for $p\ge p_{0}$
	\begin{align}
		\Vert \xi\Vert_{p}\le \sum_{k=1}^{m}C_{k}p^{\beta_{k}}+C_{m+1},\nonumber
	\end{align}
	where $C_{1},\cdots, C_{m+1}> 0$ and $\beta_{1},\cdots, \beta_{m}>0$. Then we have for any $t>0$,
	\begin{align}
		\textsf{P}\big\{ \vert \xi\vert>e(mt+C_{m+1}) \big\}\le e^{p_{0}}\exp\Big(-\min\big\{\big(\frac{t}{C_{1}}\big)^{1/\beta_{1}},\cdots, \big(\frac{t}{C_{m}}\big)^{1/\beta_{m}}\big\}\Big).\nonumber
	\end{align}
\end{mylem}

\begin{proof}
	Consider the following function,
	\begin{align}
		f(t):=\min\big\{\big(\frac{t}{C_{1}}\big)^{1/\beta_{1}},\cdots, \big(\frac{t}{C_{m}}\big)^{1/\beta_{m}}\big\}.\nonumber
	\end{align}
	If we assume $f(t)\ge p_{0}$, we can estimate 
	\begin{align}
		\Vert \xi\Vert_{f(t)}\le \sum_{k=1}^{m}C_{k}f(t)^{\beta_{k}}+C_{m+1}\le mt+C_{m+1}.\nonumber
	\end{align}
	Hence, we have by Markov's inequality
	\begin{align}
		\textsf{P}\{\vert \xi\vert>e(mt+C_{m+1})  \}\le \textsf{P}\{\vert \xi\vert >e \Vert \xi\Vert_{f(t)}  \}\le e^{-f(t)}.\nonumber
	\end{align}
	As for $f(t)<p_{0}$, we have the following trival bound
	\begin{align}
		\textsf{P}\{\vert \xi\vert>e(mt+C_{m+1})  \}\le 1\le e^{p_{0}}e^{-f(t)}.\nonumber
	\end{align}
	Hence, we have for $t\ge 0$
	\begin{align}
		\textsf{P}\{\vert \xi\vert>e(mt+C_{m+1})  \}\le  e^{p_{0}}e^{-f(t)}.\nonumber
	\end{align}
\end{proof}

We say a random variable $\xi$ has a logconcave tail if $\log\textsf{P}\{\vert \xi\vert>x  \}=x^{\alpha}$ is a concave function. Random variables with logconvex tails are similarly defined. Obviously, symmetric Weibull variables have logconcave tails when $\alpha\ge 1$ and logconvex tails when $0<\alpha\le 1$.

\begin{mylem}[Theorem 1 in \cite{Gluskin_studia_math_tail}]\label{Lem_chaos1_weibull}
	Let $\xi_{1}, \cdots, \xi_{n}\stackrel{\text{i.i.d.}}{\sim}\mathcal{W}_{s}(\alpha)$ with $\alpha\ge 1$. We have for $p\ge 1$
	\begin{align}\label{Eq_Gluskin}
		\big\Vert\sum_{i\le n}a_{i}\xi_{i}  \big\Vert_{p}\asymp_{\alpha} p^{1/2}\big(\sum_{i>p}(a_{i}^{*})^{2}\big)^{1/2}+p^{1/\alpha}\big(\sum_{i\le p}(a_{i}^{*})^{\alpha^{*}}\big)^{1/\alpha^{*}},
	\end{align}
	where $a=(a_{1},\cdots, a_{n})$ is a fixed vector, $a^{*}=(a^{*}_{i})$ is the nonincreasing rearrangement of $(\vert a_{i}\vert)$, and $\alpha^{*}=\alpha/(\alpha-1)$.
\end{mylem}
Indeed, Gluskin and Kwapien showed \eqref{Eq_Gluskin} when $\xi_{i}$ are independent random variables with logconcave tails. We only cover the symmetric Weibull case here for simplicity. We recommend interested readers to read \cite{,Adamczak_PTRF,Gluskin_studia_math_tail,Latala_studia_math_tail} to learn about similar results for random variables with logconcave tails.

\begin{mylem}[Theorem 1.1 in \cite{Hitczenko_studia}]\label{Lem_comparison_2}
	Let $\xi_{i}$ be independent symmetric random variables with logconvex tails. Then for any $p\ge 2$,
	\begin{align}
		\Vert \sum_{i=1}^{n}\xi_{i}\Vert_{p}\asymp \Big( \sum_{i=1}^{n}\textsf{E}\vert\xi_{i}\vert^{p}  \Big)^{1/p}+\sqrt{p}\Big(\sum_{i=1}^{n}\textsf{E}\xi_{i}^{2}  \Big)^{1/2}\nonumber
	\end{align}
\end{mylem}

\subsection{Contraction Principle}
In this subsection, we shall present a well-known contraction principle.  The interested reader is refered to Lemma 4.7 in \cite{Latala_Inventions} for an interesting analog.

\begin{mylem}[Lemma 4.6 in \cite{Ledoux_Talagrand_book}.]\label{Lem_contraction_principle}
	Let $F: \mathbb{R}^{+}\to \mathbb{R}^{+}$ be a convex function. Let further $\{\eta_{i}, i\le n\}$ and $\{\xi_{i}, i\le n\}$ be two symmetric sequences of independent random variables such that for some constant $K\ge 1$ and all $i\le n$ and $t>0$
	\begin{align}
		\textsf{P}\{\vert \eta_{i}\vert>t  \}\le K\textsf{P}\{\vert\xi_{i}\vert >t \}.\nonumber
	\end{align}
	Then, for any finite sequence $\{a_{i}, i\le n\}$ in a Banach space,
	\begin{align}
		\textsf{E}F\Big( \big\Vert \sum_{i=1}^{n}\eta_{i}a_{i}\big\Vert    \Big)\le \textsf{E}F\Big(K \big\Vert \sum_{i=1}^{n}\xi_{i}a_{i}\big\Vert   \Big).\nonumber
	\end{align}
	
\end{mylem}
\begin{mycor}\label{Cor_contraction_principle}
	Let $T$ be a nonempty subset of $\mathbb{R}^{n}$. $\{\eta_{i}\}$ and $\{\xi_{i}\}$ are independent random variables as in Lemma \ref{Lem_contraction_principle}. Then, we have for $p\ge 1$
	\begin{align}
		\textsf{E}\sup_{t\in T}\big\vert \sum_{i=1}^{n}\eta_{i}t_{i} \big\vert^{p}\le K^{p} \textsf{E} \sup_{t\in T}\big\vert \sum_{i=1}^{n}\xi_{i}t_{i} \big\vert^{p}.\nonumber
	\end{align}
\end{mycor}
\begin{myrem}
	To prove Corollary \ref{Cor_contraction_principle}, one can apply Lemma \ref{Lem_contraction_principle} on $(\mathbb{R}^{n}, \Vert \cdot\Vert)$. Here, the (semi)norm $\Vert \cdot\Vert$ is defined as follows:
	\begin{align}
		\Vert x\Vert=\sup_{t\in T}\big\vert \sum_{i=1}^{n}x_{i}t_{i}\big\vert,\quad x\in\mathbb{R}^{n}.\nonumber
	\end{align}
	For simplicity, we do not cover the detailed proof here. One can also deduce Corollary \ref{Cor_contraction_principle} from Lemma 4.7 in \cite{Latala_Inventions} directly.
\end{myrem}


\subsection{Comparison of Weak and Strong Moments }\label{Section_majorizing_measur_theorem}

Let $\{S_{t}=\sum t_{i}\xi : t\in T \}$ be a canonical process. There is a trivial lower estimate:
\begin{align}
	\big\Vert \sup_{t\in T} \vert S_{t}\vert\big\Vert_{p}\ge \max\Big\{ \textsf{E}\sup_{t\in T}\vert S_{t}\vert , \sup_{t\in T}\Vert S_{t}\Vert_{p}    \Big\}.\nonumber
\end{align}
In some situations, the lower bound may be reversed:
\begin{align}\label{Eq_comparison_result}
	\big\Vert \sup_{t\in T} \vert S_{t}\vert\big\Vert_{p}\lesssim  \textsf{E}\sup_{t\in T}\vert S_{t}\vert + \sup_{t\in T}\Vert S_{t}\Vert_{p}.    
\end{align}
We call a result of the form \eqref{Eq_comparison_result} a comparison of weak and strong moments. In this subsection, we shall present some such results.

An obvious case where \eqref{Eq_comparison_result} holds is that $\xi_{i}$ are normally distributed. Indeed, one can directly obtain \eqref{Eq_comparison_result} via Gaussian concentration inequality. Dilworth and Montgomery-Smith \cite{Dilworth_ap} proved the inequality \eqref{Eq_comparison_result} for $\xi_{i}$ being Rademacher random variables. Latała \cite{Latala_studia_math_tail} generalized this result to symmetric variables with logconcave tails, and Latała and Tkocz \cite{Latala_EJP} further extended it to  variables satisfying $\Vert \xi_{i}\Vert_{q}\lesssim \alpha p/q\Vert \xi\Vert_{p}$ for all $q\ge p\ge 2$. In addition to the work mentioned above, we also refer interested readers to \cite{Chen_Talagrand'sfunctional, Latala_gafa,Talagrand_chaining_book} for more information on canonical processes.

We next introduce a comparison of weak and strong moments for more general variables, which was proved by Latała and Strzelecka \cite{Latala_mathematika}.

\begin{mylem}[Theorem 1.1 in \cite{Latala_mathematika}]\label{Lem_comparison}
	Let $\xi_{1}, \cdots, \xi_{n}$ be independent centered random variables with finite moments satisfying
	\begin{align}
		\Vert \xi\Vert_{2p}\le \alpha\Vert \xi\Vert_{p}, \quad \forall p\ge 2,\quad i=1,\cdots, n,\nonumber
	\end{align}
	where $\alpha$ is a finite positive constant. Then we have for any $p\ge 1$ and any non-empty set $T\subset \mathbb{R}^{n}$
	\begin{align}
		\big\Vert \sup_{t\in T} \vert \sum_{i=1}^{n}t_{i}\xi_{i}\vert\big\Vert_{p}\lesssim_{\alpha}  \textsf{E}\sup_{t\in T}\vert \sum_{i=1}^{n}t_{i}\xi_{i}\vert + \sup_{t\in T}\Vert \sum_{i=1}^{n}t_{i}\xi_{i}\Vert_{p}. \nonumber   
	\end{align}
\end{mylem}

\section{Proof of Theorem \ref{Theo_deviation}}

\begin{proof}[Proof of Theorem \ref{Theo_deviation}]
	We first assume $\xi_{ij}\sim \mathcal{W}_{s}(\alpha)$. Note that for the symmetric matrix $X$
	\begin{align}\label{Eq_proofmain_spectral}
		\Vert X\Vert=\sup_{x\in B^{n}_{2}} \vert x^{\top}Xx\vert=\sup_{x\in B^{n}_{2}}\Big\vert \sum_{i=1}^{n}x_{i}^{2}b_{ii}\xi_{ii}+2\sum_{1\le j<i\le n}x_{i}x_{j}b_{ij}\xi_{ij}\Big\vert.
	\end{align}
	For convenience, let $$T:=\{t=(b_{11}x_{1}^{2}, 2b_{12}x_{1}x_{2},\cdots, 2b_{1n}x_{1}x_{n}, b_{22}x_{2}^{2}, \cdots, 2b_{2n}x_{2}x_{n},\cdots, b_{nn}x_{n}^{2})^{\top}: (x_{1}, \cdots,x_{n})^{\top}\in B_{2}^{n}\}$$
	and 
	\begin{align}
		\xi(X)=(\xi_{11}, \xi_{12},\cdots, \xi_{1n}, \xi_{22}, \cdots, \xi_{2n}, \cdots, \xi_{nn})^\top.\nonumber
	\end{align}
	Then, we can reformulate (\ref{Eq_proofmain_spectral}) as follows
	\begin{align}
		\Vert X\Vert=\sup_{t\in T}\vert\xi(X)^{\top}t\vert =\sup_{t\in T}\big\vert \sum t_{i}\xi(X)_{i}\big\vert.\nonumber
	\end{align}
	Note that, $\Vert\xi_{ij}\Vert_{p}\asymp p^{1/\alpha}$ for $p\ge 1$. Hence, $\Vert\xi_{ij}\Vert_{2p}\lesssim 2^{1/\alpha}\Vert \xi_{ij}\Vert_{p}$ for $p\ge 2$. It follows from Lemma \ref{Lem_comparison} for $p\ge 1$
	\begin{align}\label{Eq_bound_1}
		\Big\Vert\sup_{t\in T}\big\vert \sum t_{i}\xi(X)_{i}\big\vert \Big\Vert_{p}\lesssim_{\alpha}\textsf{E}\sup_{t\in T}\big\vert \sum t_{i}\xi(X)_{i}\big\vert+\sup_{t\in T}\big\Vert \sum t_{i}\xi(X)_{i}\big\Vert_{p}.
	\end{align}
	As for the case $\alpha\ge 1$, Lemma \ref{Lem_chaos1_weibull} yields for $p\ge 1$
	\begin{align}
		\sup_{t\in T}\big\Vert \sum t_{i}\xi(X)_{i}\big\Vert_{p}\lesssim_{\alpha} \sup_{t\in T}\Big(p^{1/2}\Vert t\Vert_{2}+ p^{1/\alpha}\Vert t\Vert_{\alpha^{*}}   \Big)\lesssim p^{1/\alpha}\max_{ij}b_{ij}.\nonumber
	\end{align}
	Next, we turn to the case $0<\alpha\le 1$. Lemma \ref{Lem_comparison_2} yields for $p\ge 2$
	\begin{align}
		\sup_{t\in T}\big\Vert \sum t_{i}\xi(X)_{i}\big\Vert_{p}\asymp 	\sup_{t\in T}\Big(p^{1/2}\Vert t\Vert_{2}+p^{1/\alpha}\Vert t\Vert_{p}\Big)\lesssim p^{1/\alpha}\max_{ij}b_{ij}.\nonumber
	\end{align}
	Hence, we have for $\alpha>0$ and $p\ge 2$
	\begin{align}
		\sup_{t\in T}\big\Vert \sum t_{i}\xi(X)_{i}\big\Vert_{p}\lesssim p^{1/\alpha}\max_{ij}b_{ij}.\nonumber
	\end{align}
	Note that, we can chosse a proper $t_{0}\in T$ such that $\vert \xi(X)^{\top}t_{0}\vert=\vert(\max_{ij}b_{ij})\xi(X)_{i_{0}}\vert$.
	Hence, we have for $p\ge 1$
	\begin{align}
		\sup_{t\in T}\big\Vert \sum t_{i}\xi(X)_{i}\big\Vert_{p}\ge \Vert\xi(X)^{\top}t_{0} \Vert_{p}\gtrsim p^{1/\alpha}\max_{ij}b_{ij}.\nonumber
	\end{align}
	Combining the following obvious bound
	\begin{align}
		\Big\Vert\sup_{t\in T}\big\vert \sum t_{i}\xi(X)_{i}\big\vert \Big\Vert_{p}\gtrsim_{\alpha}\max\Big( \textsf{E}\sup_{t\in T}\big\vert \sum t_{i}\xi(X)_{i}\big\vert, \sup_{t\in T}\big\Vert \sum t_{i}\xi(X)_{i}\big\Vert_{p}\Big)\nonumber
	\end{align}
	with \eqref{Eq_bound_1}, we have for $p\ge 2$
	\begin{align}
		\Big\Vert\sup_{t\in T}\big\vert \sum t_{i}\xi(X)_{i}\big\vert \Big\Vert_{p}\asymp_{\alpha}  \textsf{E}\Vert X\Vert+p^{1/\alpha}\max_{ij}b_{ij} ,\nonumber
	\end{align}
	
	When $\xi_{ij}$ are symmetric $\alpha$-exponential variables with parameters $\theta_{1}, \theta_{2}$, we have by Lemma \ref{Lem_moments1} for any $1\le j\le i\le n$ and $t\ge 0$
	\begin{align}
		c_{1}e^{-t^{\alpha}/c_{1}}\le \textsf{P}\{\vert\xi_{ij}\vert>t \}\le c_{2}e^{-t^{\alpha}/c_{2}},\nonumber
	\end{align}
	where $c_{1}, c_{2}$ are constants depending on $\theta_{1}, \theta_{2}$ and $\alpha$. Let $\{\eta_{ij}: 1\le j\le i\le n\}$ be a sequence of independent variables and assume further $\xi_{ij}\sim\mathcal{W}_{s}(\alpha)$. Then, we have for $t\ge 0$
	\begin{align}
		\textsf{P}\{\vert \eta_{ij}\vert>t \}\le K_{1}\textsf{P}\{ K_{1}\vert\xi_{ij}\vert>t \}, \quad \textsf{P}\{\vert\xi_{ij}\vert>t  \}\le K_{2}\textsf{P}\{K_{2}\vert \eta_{ij}\vert>t \},\,\,\,\, 1\le j\le i\le n\nonumber,
	\end{align}
	where $K_{1}=\max(1/c_{1}, 1/c_{1}^{1/\alpha}, 1)$ and $K_{2}=\max(c_{2}, c_{2}^{1/\alpha}, 1)$. Corollary \ref{Cor_contraction_principle} yields in the case where $\xi_{ij}$ are symmetric $\alpha$-exponential variables
	\begin{align}
		\Big\Vert\sup_{t\in T}\big\vert \sum t_{i}\xi(X)_{i}\big\vert \Big\Vert_{p}\asymp_{\theta_{1}, \theta_{2}, \alpha}\textsf{E}\Vert X\Vert+p^{1/\alpha}\max_{ij}b_{ij}.\nonumber
	\end{align}
	
	If $\{\xi_{ij}, i\ge j  \}$ are asymmetric, note that
	\begin{align}
		\Big\Vert\sup_{t\in T}\big\vert\xi(X)^\top t\big\vert \Big\Vert_{p}\le \Big\Vert\sup_{t\in T}\big\vert\xi(X)^\top t-\tilde{\xi}(X)^\top t\big\vert \Big\Vert_{p}\le 2\Big\Vert\sup_{t\in T}\big\vert\xi(X)^\top t\big\vert \Big\Vert_{p},\nonumber
	\end{align}
	where $\tilde{\xi}(X)$ is an independent copy of $\xi(X)$. Observe that the entries of $\tilde{\xi}(X)-\xi(X)$ are symmetric and has the same moment bounds as $\xi_{ij}$. Thus, when $\xi_{ij}$ are general $\alpha$-exponential variables, we have for $p\ge 2$
	\begin{align}\label{Eq_bound_p_th}
		C_{1} \Big(\textsf{E}\Vert X\Vert+p^{1/\alpha}\max_{ij}b_{ij}\Big)\le \Big\Vert\sup_{t\in T}\big\vert \sum t_{i}\xi(X)_{i}\big\vert \Big\Vert_{p}\le C_{2} \Big(\textsf{E}\Vert X\Vert+p^{1/\alpha}\max_{ij}b_{ij}\Big) ,
	\end{align}
	where $C_{1}< C_{2}$ are constants depending only on $\alpha, \theta_{1}$ and $\theta_{2}$. 
	
	With \eqref{Eq_bound_p_th}, we have by Lemma \ref{Lem_Moments_2}
	\begin{align}
		\textsf{P}\{ \Vert X\Vert>eC_{2}(\text{E}\Vert X\Vert+t\max_{ij}b_{ij}) \}\le e^{2}e^{-t^{\alpha}}.\nonumber
	\end{align}
	Moerover, by the Paley-Zygmund inequality, we have for $p\ge 2$
	\begin{align}
		&\textsf{P}\Big\{  \sup_{t\in T}\big\vert  \xi(X)^\top t \big\vert \ge \frac{C_{1}}{2} \Big(  \textsf{E}\Vert X\Vert+p^{1/\alpha}\max_{ij}b_{ij} \Big)   \Big\}\nonumber\\
		\ge &\textsf{P}\Big\{ \sup_{t\in T}\big\vert  \xi(X)^\top t \big\vert\ge \frac{1}{2}  \big\Vert \sup_{t\in T}\big\vert  \xi(X)^\top t \big\vert \big\Vert_{p} \Big\}\nonumber\\
		\ge & (1-2^{-p})^{2}\Big(\frac{\big\Vert \sup_{t\in T}\big\vert  \xi(X)^\top t \big\vert \big\Vert_{p}}{\big\Vert \sup_{t\in T}\big\vert  \xi(X)^\top t \big\vert \big\Vert_{2p}}   \Big)^{2p}\ge \frac{1}{2} e^{-c_{3}p},\nonumber
	\end{align}
	where $c_{3}=2\log \big(C_{2}2^{1/\alpha}/C_{1}  \big)$. If we lower bound $e^{-c_{3}p}$ by $e^{-c_{3}(p+2)}$, the inequality is valid for all $p>0$. Hence, we obtain by letting $p^{1/\alpha}=t$
	\begin{align}
		\textsf{P}\Big\{ \Vert X\Vert\ge \frac{C_{1}}{2}\Big(\textsf{E}\Vert X\Vert +t\max_{ij}b_{ij}\Big) \Big\}\ge \frac{1}{2}e^{-2c_{3}}e^{-c_{3}t^{\alpha}}.\nonumber
	\end{align} 
\end{proof}
\section{Application: The smallest singular value}
In this section, we study the smallest smallest singular value of a rectangular random matrix with independent entries drawn from $\alpha$-exponential  distribution. Consider an $N\times n$ matrix $A$  with $N\ge n$. The smallest singular value of $A$ is defined as follows:
\begin{align}
	s_{n}(A)=\inf_{x: \Vert x\Vert_{2}=1}\Vert Ax\Vert_{2}.\nonumber
\end{align}
Our main result in this section is a deviation inequality for the smallest singular value.
\begin{mytheo}\label{Theo_smallest_singular}
	Let $A$ be an $N\times n$ matrix whose entries are independent $\alpha$-exponential ($\alpha>0$) variables with parameters $\theta_{1}, \theta_{2}$. There exists $c_{1}, c_{2}>0$ and $\delta_{0}\in (0, 1)$ that depend on $\alpha, \theta_{1}, \theta_{2}$ such that if $n<\delta_{0}N$ then
	\begin{align}
		\textsf{P}\big\{s_{n}(A)\le c_{1}\sqrt{N}  \big\}\le e^{-c_{2}N^{\alpha/2}}.
	\end{align}
\end{mytheo}
Before proving this result, we first introduce a general estimate on the small ball probability and a tensorization lemma transferring one-dimensional small ball probability estimates to the multidimensional case.
\begin{mylem}[Lemma 2.6 in \cite{Rudelson_advance}]\label{Lem_small_ball}
	Let $\xi_{1}, \cdots,\xi_{n}$ be independent centered random variables with variances at least $1$ and forth moments bounded by $B$. Then there exists $\mu\in (0,1)$ depending only on $B$, such that for every coefficient vector $a=(a_{1},\cdots, a_{n})\in S^{n-1}$ the random sum $S=\sum_{k=1}^{n}a_{k}\xi_{k}$ satisfies
	\begin{align}
		\textsf{P}\{\vert S\vert<\frac{1}{2}  \}\le \mu.\nonumber
	\end{align}
\end{mylem}

\begin{mylem}[Lemma 2.2 in \cite{Rudelson_advance}]\label{Lem_tensorization}
	Let $\zeta_{1}, \cdots, \zeta_{n}$ be independent non-negative random variables. Assume that there exist $\lambda>0$ and $\mu\in (0, 1)$ such that for each $k$
	\begin{align}
		\textsf{P}\{\zeta_{k}<\lambda   \}\le \mu
	\end{align}
	Then, there exist $\lambda_{1}>0$ and $\mu_{1}\in (0, 1)$ that depend on $\lambda$ and $\mu$ only and such that
	\begin{align}
		\textsf{P}\Big\{ \sum_{k=1}^{n}\zeta_{k}^{2}<\lambda_{1}n \Big\}\le\mu_{1}^{n}.\nonumber\nonumber
	\end{align}
	
\end{mylem}
\begin{proof}[Proof of Theorem \ref{Theo_smallest_singular}]
	We first make the following decomposition:
	\begin{align}
		\textsf{P}\big\{s_{n}(A)\le c_{1}\sqrt{N}  \big\}\le \textsf{P}\big\{s_{n}(A)\le c_{1}\sqrt{N}, \Vert A\Vert \le c_{3}\sqrt{N} \big\}+\textsf{P}\{\Vert A\Vert>c_{3}\sqrt{N}  \},
	\end{align}
	where $c_{3}$ is a constant depending on $\alpha, \theta_{1}, \theta_{2}$ and will be determined below.
	
	Combining Lemma \ref{Lem_small_ball} with Lemma \ref{Lem_tensorization} yields that there exist constants $\mu_{1}, \mu_{2}\in (0, 1)$ depending only on $\theta_{1}, \theta_{2}, \alpha$ such that for every $x\in S^{n-1}$
	\begin{align}
		\textsf{P}\{\Vert Ax\Vert_{2}<\mu_{1}\sqrt{N} \}\le \mu_{2}^{N}.\nonumber
	\end{align}
	Let $\varepsilon>0$ to be chosen later. There exists an $\varepsilon$-net $\mathcal{N}$ in $S^{n-1}$ (with respect to the Euclidean norm) such that $\vert \mathcal{N}\vert\le (3/\varepsilon)^{n}$ (see e.g. Corollary 4.2.13 in \cite{Vershynin_book_highprobability}). We have by the union bound
	\begin{align}\label{Eq_smallest_proof}
		\textsf{P}\Big\{\exists x\in\mathcal{N}: \Vert Ax\Vert_{2}<\mu_{1}\sqrt{N}  \Big\}\le (3/\varepsilon)^{n}\mu_{2}^{N}.
	\end{align}
	Let $\Omega:=\Big\{\Vert A\Vert\le c_{3}\sqrt{N},\quad \Vert Ay\Vert_{2}\le \frac{1}{2}\mu_{1}\sqrt{N}, \exists y\in S^{n-1} \Big\}$. Assume that $\Omega$ occurs, and choose a point $x\in \mathcal{N}$ such that $\Vert y-x\Vert_{2}<\varepsilon$. Then, we have by setting $\varepsilon=\mu_{1}/2c_{3}$
	\begin{align}
		\Vert Ax\Vert_{2}\le \Vert Ay\Vert_{2}+\Vert A\Vert\cdot\Vert x-y\Vert_{2}\le \frac{1}{2}\mu_{1}\sqrt{N}+c_{3}\sqrt{N}\cdot \varepsilon=\mu_{1}\sqrt{N}.\nonumber
	\end{align}
	Hence, \eqref{Eq_smallest_proof} yields that
	\begin{align}
		\textsf{P}\{\Omega  \}\le (\mu_{2}(3/\varepsilon)^{n/N})^{N} \le e^{-c_{4}N}.
	\end{align}
	Here we assume that $n/N\le \delta_{0}$ for small enough $\delta_{0}$.
	
	Next we turn to bound the $\textsf{P}\{\Vert A\Vert>c_{3}\sqrt{N}  \}$.
	For an $N\times n$ asymmetric matrix $A$, set
	$$
	\hat{A}=\begin{bmatrix}
		0  &  A  \\
		A^{*}  &  0
	\end{bmatrix},$$
	where $A^{*}$ is the transpose of $A$.
	Then $\hat{A}$ is   symmetric  and  $\Vert \hat{A}\Vert=\Vert A\Vert$. In the case $0<\alpha\le 2$, we have by \eqref{Eq_optimal_bound} 
	\begin{align}\label{Eq_application1}
		\textsf{E}\Vert A\Vert \lesssim_{\alpha, \theta_{1}, \theta_{2}} \sqrt{N}.
	\end{align}
	Note that Lemma \ref{Lem_contraction_principle} yields \eqref{Eq_application1} is valid in the case $\alpha >2$.
	Then, Theorem \ref{Theo_deviation} yields that
	\begin{align}
		\textsf{P}\{\Vert A\Vert>c_{3}\sqrt{N}  \}\le e^{-c_{5}N^{\alpha/2}}.\nonumber
	\end{align}
	Hence, we conclude the proof.
\end{proof}

\section{Discussion}
In this subsection we provide an alternative interesting proof for Theorem \ref{Theo_deviation} in the case $\alpha\ge 1$. This method mainly relies on the following measure concentration phenomenon.
\begin{mylem}[Theorem 4.19 in \cite{Ledoux_book}]\label{Lem_concentration_measure}
	Let $\nu_{\alpha}$ be the measure with density $c(\alpha)e^{-\vert x\vert^{\alpha}}$, $1\le \alpha<\infty$, with respect to Lebesgue measure on $\mathbb{R}$. Denote by $\nu_{\alpha}^{n}$ be the product measure of $\nu_{\alpha}$ on $\mathbb{R}^{n}$. There is a constant $C(\alpha)>0$ only depending on $1\le \alpha<\infty$ such that for any Borel set $A$ in $\mathbb{R}^{n}$ and any $r>0$,
	\begin{align}
		1-\nu_{\alpha}^{n}\{A+\sqrt{r}B_{2}^{n}+r^{1/\alpha}B_{\alpha}^{n}  \}\le \frac{1}{\nu_{\alpha}^{n}\{A \}}e^{-r/C(\alpha)}.\nonumber
	\end{align}
\end{mylem}

\begin{proof}[Proof of Theorem \ref{Theo_deviation} in the case $\alpha\ge 1$]
	Let $\{\eta_{ij}: i\ge j\}$ be a sequence of i.i.d. random variables with density $c(\alpha)e^{-\vert x\vert^{\alpha}}$, $\alpha\ge 1$. A direct integration yields for $p\ge 1$
	\begin{align}
		\Vert \eta_{ij}\Vert_{p}=\Big(2c(\alpha)\int_{0}^{\infty}x^{p} e^{-x^{\alpha}}\, dx   \Big)^{1/p}=\Big(\frac{2c(\alpha)}{\alpha}\Gamma(\frac{p+1}{\alpha})   \Big)^{1/p}\asymp_{\alpha} p^{1/\alpha}.\nonumber
	\end{align}
	Recalling the representation \eqref{Eq_proofmain_spectral}, we have by Lemma \ref{Lem_moments1} and Corollary \ref{Cor_contraction_principle}
	\begin{align}\label{Eq_discussion_1}
		\big(\textsf{E}\Vert X\Vert^{p}\big)^{1/p}\asymp_{\theta_{1}, \theta_{2}, \alpha}\big(\textsf{E}\Vert Y\Vert^{p}\big)^{1/p}.
	\end{align}
	Here, $Y=(b_{ij}\eta_{ij})$ is an $n\times n$ symmetric matrix.
	
	Define the function $F$ as follows:
	\begin{align}
		F(u):=\sup_{x\in B^{n}_{2}}\Big\vert \sum_{i=1}^{n}x_{i}^{2}b_{ii}u_{ii}+2\sum_{1\le j<i\le n}x_{i}x_{j}b_{ij}u_{ij}\Big\vert,
	\end{align} 
	where $u=(u_{ij}: 1\le j\le i\le n)\in\mathbb{R}^{n(n+1)/2}$. Note that 
	\begin{align}
		\vert F(u)-F(v)\vert \le \sqrt{2}\max_{ij}b_{ij}\Vert u-v\Vert_{2}.\nonumber
	\end{align}
	Let $A=\{F\le m_{\Vert Y\Vert}  \}$, where $m_{\Vert Y\Vert}$ is the median of $\Vert Y\Vert$. Then, we have 
	\begin{align}
		A+\sqrt{r}B_{2}^{n(n+1)/2}+r^{1/\alpha}B_{\alpha}^{n(n+1)/2}\subset \big\{\ F\le m_{\Vert Y\Vert}+ \sqrt{2}\max_{ij}b_{ij}(\sqrt{r}+r^{1/\alpha})\big \}.\nonumber
	\end{align}
	Note that $F(\eta)=\Vert Y\Vert$, where $\eta=(\eta_{ij}: 1\le j\le i\le n)$. Then, we have by Lemma \ref{Lem_concentration_measure}
	\begin{align}
		\textsf{P}\{\Vert Y\Vert>m_{\Vert Y\Vert}+\sqrt{2}\max_{ij}b_{ij}(\sqrt{r}+r^{1/\alpha})  \}\le 2e^{-r/C(\alpha)}.\nonumber
	\end{align}
	Similarly, we have 
	\begin{align}
		\textsf{P}\{\Vert Y\Vert\le m_{\Vert Y\Vert}-\sqrt{2}\max_{ij}b_{ij}(\sqrt{r}+r^{1/\alpha})  \}\le 2e^{-r/C(\alpha)}.\nonumber
	\end{align}
	Integration yields for $p\ge 1$
	\begin{align}
		(\textsf{E}\Vert Y\Vert^{p})^{1/p}\lesssim_{\alpha} m_{\Vert Y\Vert}+\max_{ij}b_{ij}p^{1/\alpha}.\nonumber
	\end{align}
	Note that
	\begin{align}
		\big\vert\textsf{E}\Vert Y\Vert-m_{\Vert Y\Vert}\big\vert\le\textsf{E}\big\vert \Vert Y\Vert-m_{\Vert Y\Vert} \big\vert \le C_{1}(\alpha)\max_{ij}b_{ij}.\nonumber
	\end{align}
	Hence, we have $p\ge 1$
	\begin{align}
		(\textsf{E}\Vert Y\Vert^{p})^{1/p}\lesssim_{\alpha} \textsf{E}\Vert Y\Vert+\max_{ij}b_{ij}p^{1/\alpha}.\nonumber
	\end{align}
	One can obtain the lower bound following the same line as in the proof of Theorem $\ref{Theo_deviation}$. By virtue of \eqref{Eq_discussion_1}, we have 
	\begin{align}
		(\textsf{E}\Vert X\Vert^{p})^{1/p}\asymp_{\theta_{1}, \theta_{2}, \alpha} \textsf{E}\Vert X\Vert+\max_{ij}b_{ij}p^{1/\alpha}.\nonumber
	\end{align}
	The desired result follows from the Markov inequality and the Paley-Zygmund inequality.
	
\end{proof}

\textbf{Acknowledgment} Su was partly supported by the National Natural Science Foundation of China (12271475, U23A2064) and fundamental research funds for central university grants.

\end{document}